\documentclass[a4paper,11pt]{article}
\usepackage[T1]{fontenc}
\usepackage[utf8]{inputenc}
\usepackage[english]{babel}
\usepackage{microtype}
\usepackage[margin=1.1in]{geometry}
\usepackage{braket}
\usepackage{amsmath}
\usepackage{amssymb}
\usepackage{amsthm, aliascnt}
\usepackage{mathtools}
\usepackage{mathrsfs}
\usepackage{xcolor}
\usepackage{hyperref}
\usepackage{enumitem}
\usepackage{tikz}
\usepackage{hyperref}
\usepackage[maxbibnames=99]{biblatex}
\usepackage{csquotes}
\usepackage{esint}%
\usetikzlibrary{patterns}

\hypersetup{
                colorlinks=true,
                linkcolor={magenta},
                citecolor={cyan},}

\theoremstyle{plain}
\newtheorem{teor}{Theorem}
\numberwithin{teor}{section}
\numberwithin{equation}{section}
\theoremstyle{definition}
\newaliascnt{defi}{teor}
\newtheorem{defi}[defi]{Definition}
\aliascntresetthe{defi}

\theoremstyle{plain}
\newaliascnt{lemma}{teor}%
\newtheorem{lemma}[lemma]{Lemma}
\aliascntresetthe{lemma}
\theoremstyle{plain}
\newaliascnt{prop}{teor}%
\newtheorem{prop}[prop]{Proposition}
\aliascntresetthe{prop}
\theoremstyle{plain}
\newaliascnt{cor}{teor}%

\aliascntresetthe{cor}
\theoremstyle{definition}
\newaliascnt{ex}{teor}%

\aliascntresetthe{ex}
\theoremstyle{remark}
\newaliascnt{oss}{teor}%
\newtheorem{oss}[oss]{Remark}
\aliascntresetthe{oss}
\DeclarePairedDelimiter{\abs}{\lvert}{\rvert}
\DeclarePairedDelimiter{\norma}{\lVert}{\rVert}
\DeclareMathOperator{\sbv}{SBV}
\DeclareMathOperator{\bv}{BV}
\newcommand{\marcomm}[1]{\marginpar{\begin{flushright}#1\end{flushright}}}%
\newcommand{\R}{\mathbb{R}}

\newcommand{\N}{\mathbb{N}}
\newcommand{\Ln}{\mathcal{L}^n}
\newcommand{\Hn}{\mathcal{H}^{n-1}}
\newcommand{\eps}{\varepsilon}

\DeclareMathOperator{\divv}{div}
\DeclareMathOperator{\loc}{loc}
\DeclareMathOperator{\supp}{supp}

\makeatletter
\newcommand{\leqnomode}{\tagsleft@true\let\veqno\@@leqno}
\newcommand{\reqnomode}{\tagsleft@false\let\veqno\@@eqno}

\makeatother

\title{A free boundary problem for the $p$-Laplacian with nonlinear boundary conditions}
\author{P. Acampora, E. Cristoforoni }

\begin{document}
\renewcommand*{\sectionautorefname}{Section}
\newcommand{\Addresses}{{%
 \bigskip 
 \footnotesize 
 
 \textsc{Dipartimento di Matematica e Applicazioni ``R. Caccioppoli'', Universit\`a degli studi di Napoli Federico II, Via Cintia, Complesso Universitario Monte S. Angelo, 80126 Napoli, Italy.}\par\nopagebreak 
 
 \medskip 
 
 \textit{E-mail address}, P.~Acampora: \texttt{paolo.acampora@unina.it} 
 \medskip 
 
\textsc{Mathematical and Physical Sciences for Advanced Materials and Technologies, Scuola Superiore Meridionale, Largo San Marcellino 10, 80126, Napoli, Italy.}\par\nopagebreak 
 
 \medskip 
 
 \textit{E-mail address}, E.~Cristoforoni: \texttt{emanuele.cristoforoni@unina.it} 
}} 
\reversemarginpar
\maketitle
\begin{abstract}
We study a nonlinear generalization of a free boundary problem that arises in the context of thermal insulation. We consider two open sets $\Omega\subseteq A$, and we search for an optimal $A$ in order to minimize a non-linear energy functional, whose minimizers $u$ satisfy the following conditions: $\Delta_p u=0$ inside $A\setminus\Omega$, $u=1$ in $\Omega$, and a nonlinear Robin-like boundary $(p,q)$-condition on the free boundary $\partial A$. We study the variational formulation of the problem in $\sbv$, and we prove that, under suitable conditions on the exponents $p$ and $q$, a minimizer exists and its jump set satisfies uniform density estimates.


\end{abstract}
\section{Introduction}
Let $\Omega\subseteq\R^n$ be a bounded open set with smooth boundary, and let $A$ be a set containing $\Omega$. Consider the functional
\begin{equation}\label{0}
F(A,v)=\int_{A}\abs{\nabla v}^2\,d\Ln+\beta\int_{\partial A}v^2\,d\Hn+C_0\Ln(A),
\end{equation}
with $v\in H^1(A)$, $v=1$ in $\Omega$ and $\beta, C_0>0$ fixed positive constants. The problem of minimizing this functional arises in the environment of thermal insulation: $F$ represents the energy of a heat configuration $v$ when the temperature is maintained constant inside the body $\Omega$ and there's a bulk layer $A\setminus \Omega$ of insulating material whose cost is represented by $C_0$ and the heat transfer with the external environment  is conveyed by convection. For simplicity's sake in the following we will set $C_0=1$. The variational formulation in \eqref{0} leads to an Euler-Lagrange equation, which is the weak form of the following problem:
\begin{equation}\label{eq0}\begin{cases}
\Delta u=0 & \text{in } A\setminus\Omega, \\[3 pt]
\dfrac{\partial u}{\partial \nu} +\beta u=0 & \text{on } \partial A,\\[4 pt]
u= 1 & \text{in }\Omega,
\end{cases}\end{equation}\medskip

The problems we are interested in concern the existence of a solution and its regularity. In this sense, one could be interested in studying a more general setting in which it is possible to consider possibly irregular sets $A$. Specifically, we could generalize the problem into the context of $\sbv$ functions, aiming to minimize the functional
\[
F(v)=\int_{\R^n}\abs{\nabla v}^2\, d\Ln+\beta\int_{J_v}\left(\underline{v}^2+\overline{v}^2\right)\,d\Hn+\Ln(\set{v>0}\setminus\Omega)
\]
with $v\in\sbv(\R^n)$ and $v=1$ in $\Omega$. This problem has been studied by L. A. Caffarelli and D. Kriventsov in \cite{CK}, where the authors have proved the existence of a solution $u$ for the problem and the regularity of its jump set. Another similar problem, in a non-linear context, has been deepened by D. Bucur and A. Giacomini in \cite{BucGia15} with a boundedness constraint. \medskip

In this paper, our main aim is to generalize the problem and  techniques employed in \cite{CK} to a nonlinear formulation. In detail, for $p,q>1$ fixed, we consider the functional
\begin{equation}\label{funzionale}
\mathcal{F}(v)=\int_{\R^n}\abs{\nabla v}^p\,d\Ln+\beta\int_{J_v}\left(\underline{v}^q+\overline{v}^q\right)\, d\Hn +\Ln(\set{v>0}\setminus\Omega),
\end{equation}
and in the following we are going to study the problem
\[%
\inf\Set{\mathcal{F}(v) | \begin{gathered}
v\in\sbv(\R^n) \\
v(x)=1 \text{ in }\Omega 
\end{gathered}}.
\]%
Notice that if $v\in\sbv(\R^n)$ with $v=1$ a.e. in $\Omega$, letting $v_0=\max\set{0,\min\set{v,1}}$ we have that $v_0\in\sbv(\R^n)$ with $v_0=1$ a.e. in $\Omega$ and $\mathcal{F}(v_0)\le\mathcal{F}(v)$ so it suffices to consider the problem
\begin{equation}
\label{problema}
\inf\Set{\mathcal{F}(v) | \begin{gathered}
v\in\sbv(\R^n),\\ v(x)\in[0,1] \,\Ln\text{-a.e.}, \\
v(x)=1 \text{ in }\Omega 
\end{gathered}}.
\end{equation}
In a more regular setting, problem \eqref{problema} can be seen as a PDE. Let us fix $\Omega,A$ sufficiently smooth open sets, $u\in W^{1,p}(A)$ with $u=1$ on $\Omega$, and let us define the functional
\begin{equation}\label{eq: 0}
F(u,A)=\int_\Omega \abs{\nabla u}^p\, d\Ln+\beta\int_{\partial\Omega}\abs{u}^q\,d\Hn + \Ln({A\setminus\Omega}).
\end{equation}
minimizers u to~\eqref{eq: 0} solve the following boundary value problem 
\begin{equation}\label{eq1}
\begin{dcases}
\vphantom{\Big(}\divv\left(\abs{\nabla u}^{p-2}\nabla u\right)=0 \qquad \qquad&\text{in $A\setminus\Omega$,}\\[7 pt]
\vphantom{\bigg(}\abs{\nabla u}^{p-2}\frac{\partial u}{\partial \nu}+\beta\frac{q}{p}\abs{u}^{q-2}u=0 &\text{on $\partial A$,}\\[7 pt]
u=1 &\text{in $\Omega$.}
\end{dcases}
\end{equation}
In \autoref{notation} we give some preliminary tools and definitions, and then we will prove the existence of a minimizer $u$ of \eqref{problema}, under a prescribed condition on $p$ and $q$. Finally, we will prove density estimates for the jump set $J_u$. \medskip

We resume in the following theorems the main results of this paper.
\begin{teor}
\label{teor: main1}
Let $\Omega\subseteq\R^n$ be a bounded open set, and let $p,q>1$ be exponents satisfying one of the following conditions:
\begin{itemize}
\item{ $1<p<n$, and 
$1<q<\dfrac{p(n-1)}{n-p}:=p_*$;}
\item{$n\le p<\infty$, and $1<q<\infty$.}
\end{itemize}
Then there exists a solution $u$ to problem \eqref{problema} and there exists a constant $\delta_0=\delta_0(\Omega,\beta,p,q)>0$ such that
\begin{equation}\label{eq: lower bound0}
u>\delta_0\end{equation}
$\Ln$-almost everywhere in $\set{u>0}$, and there exists $\rho(\delta_0)>0$ such that
\[\supp u\subseteq B_{\rho(\delta_0)}.
\]
\end{teor}
\begin{teor}
\label{teor: main2}
Let $\Omega\subseteq\R^n$ be a bounded open set, and let $p,q>1$ be exponents satisfying the assumptions of \autoref{teor: main1}. Then there exist positive constants $C(\Omega,\beta,p,q)$, $c(\Omega,\beta,p,q)$, $C_1(\Omega,\beta,p,q)$ such that if $u$ is a minimizer to problem \eqref{problema}, then
\[
c\,r^{n-1}\le\Hn(J_u\cap B_r(x))\le C\, r^{n-1},
\]
and
\[
\Ln(B_r(x)\cap\set{u>0})\ge C_1\,r^n,
\]
for every $x\in\overline{J_u}$ with $B_r(x)\subseteq \R^n\setminus \Omega$.

In particular, this implies the essential closedness of the jump set $J_u$, namely
\[
\Hn(\overline{J_u}\setminus J_u)=0.
\]
\end{teor}
In \autoref{section: lower} we prove that the a priori estimate~\eqref{eq: lower bound0} holds for inward minimizers (see \autoref{defi: inward}), such an estimate will be crucial in the   proof of \autoref{teor: main1} in \autoref{section: exist}. Finally, in \autoref{section: density} we prove \autoref{teor: main2}.
\begin{oss}
Notice that the condition on the exponents is undoubtedly verified when $p\ge q>1$.
Furthermore, if $\Omega$ is a set with Lipschitz boundary, the exponent $p_*$ is the optimal exponent such that
\[
W^{1,p}(\Omega)\subset\subset L^q(\partial\Omega) \qquad \forall q\in[1,p_*).
\]
\end{oss}

\section{Notation and Tools}\label{notation}
In this section, we give the definition of the space $\sbv$, and some useful notations and results that we will use in the following sections. We refer to \cite{bv}, \cite{evans}, \cite{robin-bg} for a more intensive study of these topics.
\begin{defi}[$\bv$]
Let $u\in L^1(\R^n)$. We say that $u$ is a function of \emph{bounded variation} in $\R^n$ and we write $u\in\bv(\R^n)$ if its distributional derivative is  a Radon measure, namely
\[
\int_{\Omega}u\,\frac{\partial\varphi}{\partial x_i}=\int_{\Omega}\varphi\, d D_i u\qquad \forall \varphi\in C^\infty_c(\R^n),
\]
with $Du$ a $\R^n$-valued measure in $\R^n$. We denote with $\abs{Du}$ the total variation of the measure $Du$. The space $\bv(\R^n)$ is a Banach space equipped with the norm
\[
\norma{u}_{\bv(\R^n)}=\norma{u}_{L^1(\R^n)}+\abs{Du}(\R^n).
\]

\end{defi}
\begin{defi}
Let $E\subseteq\R^n$ be a measurable set. We define the \emph{set of points  of density 1 for $E$} as 
\[
E^{(1)}=\Set{x\in\R^n | \lim_{r\to0^+}\dfrac{\Ln(B_r(x)\cap E)}{\Ln(B_r(x))}=1},
\]
and the \emph{set of points of density 0 for $E$} as 
\[
E^{(0)}=\Set{x\in\R^n | \lim_{r\to0^+}\dfrac{\Ln(B_r(x)\cap E)}{\Ln(B_r(x))}=0}.
\]
Moreover, we define the \emph{essential boundary} of $E$ as
\[
\partial^*E=\R^n \setminus(E^{(0)}\cup E^{(1)}).
\]
\end{defi}
\begin{defi}[Approximate upper and lower limits]
Let $u\colon\R^n\to\R$ be a measurable function. We define the \emph{approximate upper and lower limits} of $u$, respectively, as
\[\overline{u}(x)=\inf\Set{t\in\R|\limsup_{r\to0^+}\dfrac{\Ln(B_r(x)\cap\set{u>t})}{\Ln(B_r(x))}=0},\]
and
\[\underline{u}(x)=\sup\Set{t\in\R|\limsup_{r\to0^+}\dfrac{\Ln(B_r(x)\cap\set{u<t})}{\Ln(B_r(x))}=0}.\]
We define the \emph{jump set} of $u$ as 
\[J_u=\Set{x\in\R^n|\underline{u}(x)<\overline{u}(x)}.\] 
We denote by $K_u$ the closure of $J_u$. 
\end{defi} 
If $\overline{u}(x)=\underline{u}(x)=l$, we say that $l$ is the approximate limit of $u$ as $y$ tends to $x$, and we have that, for any $\eps>0$, 
\[\limsup_{r\to0^+}\dfrac{\Ln(B_r(x)\cap\set{\abs{u-l}\geq\eps)}}{\Ln(B_r(x))}=0.\]

If $u\in\bv(\R^n)$, the jump set $J_u$ is a $(n-1)$-rectifiable set, i.e. ${J_u\subseteq\bigcup_{i\in\mathbb{N}}M_i}$, up to a $\Hn$-negligible set, with $M_i$ a $C^1$-hypersurface in $\R^n$ for every $i$. We can then define $\Hn$-almost everywhere on $J_u$ a normal $\nu_u$ coinciding with the normal to the hypersurfaces $M_i$. Furthermore, the direction of $\nu_u(x)$ is chosen in such a way that the approximate upper and lower limits of $u$ coincide with the approximate limit of $u$ on the half-planes
\[H^+_{\nu_u}=\set{y\in\R^n|\nu_u(x)\cdot(y-x)\geq0}\]
and
\[H^-_{\nu_u}=\set{y\in\R^n|\nu_u(x)\cdot(y-x)\leq0}\]
respectively.
\begin{defi}
Let $\Omega\subseteq\R^n$ be an open set, and $E\subseteq\R^n$ a measurable set. We define the \emph{relative perimeter} of $E$ inside $\Omega$ as
\[
P(E;\Omega)=\sup\Set{\int_E \divv\varphi\,d\Ln | \begin{aligned}
\varphi\in &\:C^1_c(\Omega,\R^n) \\ &\abs{\varphi}\le 1
\end{aligned}}.
\]
If $P(E;\R^n)<+\infty$ we say that $E$ is a \emph{set of finite perimeter}.
\end{defi}

\begin{teor}[Decomposition of $\bv$ functions]
Let $u\in\bv(\R^n)$. Then we have
\[
dDu=\nabla u\,d\Ln+\abs{\overline{u}-\underline{u}}\nu_u\,d\Hn\lfloor_{{\Huge J_u}}+ dD^c u,
\]
where $\nabla u$ is the density of $Du$ with respect to the Lebesgue measure, $\nu_u$ is the normal to the jump set $J_u$ and $D^c u$ is the \emph{Cantor part} of the measure $Du$. The measure $D^c u$ is singular with respect to the Lebesgue measure and concentrated out of $J_u$.
\end{teor}
\begin{defi}
Let $v\in\bv(\R^n)$, let $\Gamma\subseteq\R^n$ be a $\Hn$-rectifiable set, and let $\nu(x)$ be the generalized normal to $\Gamma$ defined for $\Hn$-a.e. $x\in\Gamma$. For $\Hn$-a.e. $x\in \Gamma$ we define the traces $\gamma_\Gamma^{\pm}(v)(x)$ of $v$ on $\Gamma$ by the following Lebesgue-type limit quotient relation
\[
\lim_{r\to 0}\frac{1}{r^n}\int_{B_r^{\pm}(x)}\abs{v(y)-\gamma_\Gamma^{\pm}(v)(x)}\,d\Ln(y)=0,
\]
where
\[
B_{r}^{+}(x)=\set{y\in B_r(x) | \nu(x)\cdot(y-x)>0},
\]
\[
B_{r}^{-}(x)=\set{y\in B_r(x) | \nu(x)\cdot(y-x)<0}.
\]\end{defi}
\begin{oss}
Notice that, by~\cite[Remark 3.79]{bv}, %
for $\Hn$-a.e. $x\in\Gamma$, $(\gamma_\Gamma^{+}(v)(x),\gamma_\Gamma^-(v)(x))$ coincides with either  $(\overline{v}(x),\underline{v}(x))$ or $(\underline{v}(x),\overline{v}(x))$, while, for $\Hn$-a.e. $x\in \Gamma\setminus J_v$, we have that $\gamma_\Gamma^+(v)(x)=\gamma_\Gamma^-(v)(x)$ and they coincide with the approximate limit of $v$ in $x$. In particular, if $\Gamma=J_v$, we have
\[
\gamma_{J_v}^+(v)(x)=\overline{v}(x) \qquad \gamma_{J_v}^-(v)(x)=\underline{v}(x)
\]
for $\Hn$-a.e. $x\in J_v$. 

\end{oss}
We now focus our attention on the $\bv$ functions whose Cantor parts vanish.
\begin{defi}[$\sbv$]
Let $u\in\bv(\R^n)$. We say that $u$ is a \emph{special function of bounded variation} and we write $u\in\sbv(\R^n)$ if $D^c u=0$. %
\end{defi}
For $\sbv$ functions we have the following.
\begin{teor}[Chain rule]\label{teor: chain}
Let $g\colon\R\to\R$ be a differentiable function. Then if $u\in\sbv(\R^n)$, we have
\[\nabla g(u)=g'(u)\nabla u.\]
Furthermore, if $g$ is increasing,
\[\overline{g(u)}=g(\overline{u}),\quad \underline{g(u)}=g(\underline{u})\]
 while, if $g$ is decreasing,
\[\overline{g(u)}=g(\underline{u}),\quad \underline{g(u)}=g(\overline{u}).\]
\end{teor}
We now state a compactness theorem in $\sbv$ that will be useful in the following. 
\begin{teor}\label{teor:compact}[Compactness in $\sbv$]
Let $u_k$ be a sequence in $\sbv(\R^n)$. Let $p,q>1$, and let $C>0$ such that for every $k\in\N$ 
\[
\int_{\R^n}\abs{\nabla u_k}^p\,d\Ln+\norma{u_k}_{\infty}+\Hn(J_{u_k})<C.
\]
Then there exists $u\in\sbv(\R^n)$ and a subsequence $u_{k_j}$ such that
\begin{itemize}
    \item \emph{Compactness:}
    \[
    u_{k_j}\xrightarrow{L^1_{\loc}(\R^n)} u
    \]
    \item \emph{Lower semicontinuity:} for every open set $A$ we have
    \[
    \int_A \abs{\nabla u}^p\, d\Ln  \le \liminf_{j\to+\infty}\int_A \abs{\nabla u_{k_j}}^p\,d\Ln
    \]
    and
    \[
    \int_{J_u\cap A}\left( \overline u^q+\underline u^q\right)\, d\Hn  \le \liminf_{j\to+\infty}\int_{J_{u_{k_j}}\cap A} \left(\overline u_{k_j}^q+\underline u_{k_j}^q\right)\, d\Hn
    \]
\end{itemize}
\end{teor}
We refer to \cite[Theorem 4.7, Theorem 4.8, Theorem 5.22]{bv} for the proof of this theorem. We now conclude this section with the following proposition whose proof can be found in \cite[Lemma 3.1]{CK}.
\begin{prop}
Let $u\in\bv(\R^n)\cap L^\infty(\R^n)$. Then
\[
\int_0^1 P(\set{u>s};\R^n\setminus J_u)\,ds=\abs{Du}(\R^n\setminus J_u).
\]
\end{prop}
\section{Lower Bound}
\label{section: lower}
In the following, we assume that $\Omega\subset\R^n$ is a bounded open set and that $p$ and $q$ are two positive real numbers such that
\begin{equation}\label{eq:p/q}\dfrac{q'}{p'}>1-\dfrac{1}{n}\end{equation}
where $p'$ and $q'$ are the Hölder conjugates of $p$ and $q$ respectively.

\begin{defi}
\label{defi: inward}
Let $v\in\sbv(\R^n)$ be a function such that $v=1$ a.e. in $\Omega$. We say that $v$ is an \emph{inward minimizer} if
\[\mathcal{F}(v)\leq \mathcal{F}(v\chi_A),\]
for every set of finite perimeter $A$ containing $\Omega$, where $\chi_A$ is the characteristic function of set $A$.
\end{defi}

Let $A\subset \R^n$ be a set of finite perimeter such that $\Omega\subset A$, and let $v\in\sbv(\R^n)$. We will make use of the following expression
\begin{equation}
\label{eq: truncated}
\begin{split}
\mathcal{F}(v\chi_A)=&\int_{A}\abs{\nabla v}^p\, d\Ln + \beta\int_{J_v\cap A^{(1)}}\left(\underline{v}^q+\overline{v}^{\,q}\right) \,d\Hn +\beta\int_{\partial^* A\setminus J_v}v^q\,d\Hn\\[5 pt] &+\beta\int_{J_v\cap\partial^* A}\gamma_{\partial A}^-(v)^q\, d\Hn +\Ln\left((\set{v>0}\cap A)\setminus\Omega\right),
\end{split}
\end{equation}
Let $B$ be a ball containing $\Omega$, then $\chi_B\in\sbv(\R^n)$ and $\chi_B=1$ in $\Omega$, we will denote $\mathcal{F}(\chi_B)$ by $\tilde{\mathcal{F}}$.

\begin{teor}\label{teor: lower-bound}
There exists a positive constant $\delta=\delta(\Omega,\beta,p,q)$ such that if $u$ is an inward minimizer with $\mathcal{F}(u)\leq2\tilde{\mathcal{F}}$, then \[u>\delta\] 
$\Ln$-almost everywhere in $\Set{u>0}$.
\end{teor}
\begin{proof}
Let $0<t<1$ and
\[f(t)=\int_{\set{u\leq t}\setminus J_u} u^{q-1}\abs{\nabla u}\,d\Ln=\int_0^t s^{q-1}P(\Set{u> s};\R^n\setminus J_u)\,ds.\]
For every such $t$, we have
\begin{equation}
\label{eq: fbound}
    f(t)\leq\left(\int_{\set{u\leq t}} u^{(q-1)p'}\,d\Ln\right)^{\frac{1}{p'}}\left(\int_{\set{u\leq t}\setminus J_u}\abs{\nabla u}^p\,d\Ln\right)^{\frac{1}{p}}\leq \mathcal{F}(u)\leq 2\tilde{\mathcal{F}}.
\end{equation}
Let $u_t=u\chi_{\set{u>t}}$. Using \eqref{eq: truncated} we have
\[\begin{split}0\leq& \,\mathcal{F}(u_t)-\mathcal{F}(u)\\[5 pt]
=&\beta\int_{\partial^*\set{u>t}\setminus J_u}\overline{u}^q\,d\Hn-\int_{\set{u\leq t}\setminus J_u}\abs{\nabla u}^p\,d\Ln-\beta\int_{J_u\cap\partial^*\set{u>t}}\underline{u}^q\,d\Hn+\\[5 pt]&-\beta\int_{J_u\cap\set{u>t}^{(0)}}\left(\overline{u}^q+\underline{u}^q\right)\,d\Hn-\Ln(\set{0<u\leq t}),
\end{split}\]
and rearranging the terms,
\begin{equation}
\label{eq: derestimates}
\begin{split}\int_{\set{u\leq t}\setminus J_u}\abs{\nabla u}^p\,d\Ln&+\!\beta\int_{J_u\cap\partial^*\set{u>t}}\underline{u}^q\,d\Hn+\!\beta\int_{J_u\cap\set{u>t}^{(0)}}\left(\overline{u}^q+\underline{u}^q\right)\,d\Hn+\\[7 pt]&+\Ln(\set{0<u\leq t})\leq \beta t^q P(\set{u>t};\R^n\setminus J_u)=\beta t f'(t). \end{split}
\end{equation}
On the other hand,
\[\begin{split}f(t)&=\int_{\set{u\leq t}\setminus J_u} u^{q-1}\abs{\nabla u}\,d\Ln\\[5 pt] &\leq \left(\int_{\set{u\leq t}} u^{(q-1)p'}\,d\Ln\right)^{\frac{1}{p'}}\!\!\left(\int_{\set{u\leq t}\setminus J_u}\abs{\nabla u}^p\,d\Ln\right)^{\frac{1}{p}}\\[7 pt]
&\leq\! \Biggl(\Ln(\set{0<u\leq t})\Biggr)^{\!\!\frac{1}{p'\gamma'}}\!\!\left(\int_{\set{u\leq t}} u^{q 1^*}\,d\Ln\right)^{\!\!\frac{1}{q'1^*}}\!\!\left(\int_{\set{u\leq t}\setminus J_u}\abs{\nabla u}^p\,d\Ln\right)^{\frac{1}{p}},
\end{split}\]
where we used \[1^*=\dfrac{n}{n-1},
\qquad \text{and}\qquad 
\gamma=\dfrac{q1^*}{(q-1)p'},  \]
and $\gamma>1$ by \eqref{eq:p/q}. By classical BV embedding in $L^{1^*}$ applied to the function $(u\chi_{\set{u\leq t}})^q$ and the estimate \eqref{eq: derestimates}, we have
 \[f(t)\leq C(n,\beta) \biggl(t f'(t)\biggr)^{1-\frac{n-1}{q'n}}\left(\int_{\R^n}\,d\abs*{D (u\chi_{\set{u\leq t}})^q}\right)^{\frac{1}{q'}}.\]
We can compute
\[\begin{split}\int_{\R^n}\,d\abs*{D(u\chi_{\set{u\leq t}})^q}&\leq q\Biggl(\Ln(\set{0<u\leq t})\Biggr)^{\frac{1}{p'}}\left(\int_{\set{u\leq t}\setminus J_u}\abs{\nabla u}^p\,d\Ln\right)^{\frac{1}{p}}+\\[5 pt]&+\! \int_{J_u\cap\set{u>t}^{(0)}}\left(\overline{u}^q+\underline{u}^q\right)\,d\Hn+\!\int_{J_u\cap \partial^*\set{u>t}}\underline{u}^q\,d\Hn+\\[5 pt]&+t^qP(\set{u>t};\R^n\setminus J_u)\leq (2+q\beta)tf'(t).
\end{split}\]
We therefore get
\[f(t)\leq C(n,\beta,q) \left(tf'(t)\right)^{1+\frac{1}{nq'}}.\]
Let $0<t_0<1$ such that $f(t_0)>0$, then for every $t_0<t<1$, we have $f(t)>0$ and
\[\dfrac{f'(t)}{f(t)^{\frac{nq}{q(n+1)-1}}}\geq \dfrac{C(n,\beta,q)}{t},\]
integrating from $t_0$ to $1$, we have
\[f(1)^{\frac{q-1}{q(n+1)-1}}-f(t_0)^{\frac{q-1}{q(n+1)-1}}\geq C(n,\beta,q) \log\dfrac{1}{t_0},\]
so that, using \eqref{eq: fbound},
\[f(t_0)^{\frac{q-1}{q(n+1)-1}}\leq (2\tilde{\mathcal{F}})^{\frac{q-1}{q(n+1)-1}} + C(n,\beta,q)\log t_0.\]
Let \[\delta=\exp\left(-\dfrac{(2\tilde{\mathcal{F}})^{\frac{q-1}{q(n+1)-1}}}{C(n,\beta,q)}\right),\]
for every $t_0<\delta$ we would have $f(t_0)<0$, which is a contradiction. Therefore $f(t)=0$ for every $t<\delta$, from which $u>\delta$ $\Ln$-almost everywhere on $\set{u>0}$.
\end{proof}
\begin{oss}
\label{oss: incljump}
From \autoref{teor: lower-bound}, if $u$ is an inward minimizer with $\mathcal{F}(u)\le2\tilde{\mathcal{F}}$,  we have that 
\begin{equation*}
\partial^*\set{u>0}\subseteq J_u\subseteq K_u.
\end{equation*} Indeed, on $\partial^*\set{u>0}$ we have that, by definition,  $\underline{u}=0$ and that, since $u\ge\delta$ $\Ln$-almost everywhere in $\set{u>0}$, $\overline{u}\ge\delta$.
\end{oss}
\begin{prop}\label{cor: density0}
There exists a positive constant $\delta_0=\delta_0(\Omega,\beta,p,q)<\delta$ such that if $u$ is an inward minimizer with $\mathcal{F}(u)\leq2\tilde{\mathcal{F}}$, then $u$ is supported on $B_{\rho(\delta_0)}$, where $\rho(\delta_0)=\delta_0^{1-q}$ and $B_{\rho(\delta_0)}$ is the ball centered at the origin with radius $\rho(\delta_0)$. Moreover there exist positive constants $C(\Omega,\beta,p,q),C_1(\Omega,\beta,p,q)$ such that, for any $B_r(x)\subseteq \R^n\setminus\Omega$ we have
\begin{equation}\label{eq: density up}\Hn(J_u\cap B_r(x))\leq C(\Omega,p,q)r^{n-1},\end{equation}
and if $x\in K_u$, then
\begin{equation}\label{eq: densità supp}\Ln(B_r(x)\cap\set{u>0})\geq C_1(\Omega,p,q)r^n.\end{equation}
\end{prop}
\begin{proof}
By \autoref{teor: lower-bound}, if $u$ is an inward minimizer, we have
\[\int_{J_u\cap B_r(x)}\left(\overline{u}^q+\underline{u}^q\right)\,d\Hn\geq\delta^q\Hn(J_u\cap B_r(x)),\]
on the other hand, using $u\chi_{\R^n\setminus B_r(x)}$ as a competitor for $u$, we have 
\[\int_{J_u\cap B_r(x)}\left(\overline{u}^q+\underline{u}^q\right)\,d\Hn\leq \int_{\partial B_r(x)\cap\set{u>0}^{(1)}}\left(\overline{u}^q+\underline{u}^q\right)\,d\Hn\leq C(n)r^{n-1}.\]

Let now $x\in K_u$ and consider $\mu(r)=\Ln\left(B_r(x)\cap\set{u>0}^{(1)}\right)$. Using the isoperimetric inequality and inequality~\eqref{eq: density up}, we have that for almost every $r\in(0,d(x,\Omega))$
\[\begin{split}0<\mu(r)&\leq K(n)\,P\!\left(B_r(x)\cap\set{u>0}^{(1)}\right)^{\frac{n}{n-1}}\\&\leq K(\Omega,\beta,p,q)\,P\!\left(B_r(x);\set{u>0}^{(1)}\right)^{\frac{n}{n-1}}.\end{split}\]
Notice that we used \autoref{oss: incljump} in the last inequality. We have
\[\mu(r)\leq K \mu'(r)^{\frac{n}{n-1}}.\]
Integrating the differential inequality, we obtain
\[\Ln(B_r(x)\cap\set{u>0})\geq C_1(\Omega,\beta,p,q)r^n.\]
Finally, let $\delta_0>0$ and $x\in K_u$ such that $d(x,\Omega)>\rho(\delta_0)=\delta_0^{1-q}$. By~\eqref{eq: densità supp}
\[C_1(\Omega,\beta,p,q)\rho(\delta_0)^{n}\leq\Ln(\set{u>0}\cap\Omega)\leq2\tilde{\mathcal{F}},\]
which leads to a contradiction if $\delta_0$ is too small, hence there exists a positive value $\delta_0(\Omega,\beta,p,q)$ such that $\set{u>0}\subset B_{\rho(\delta_0)}$.
\end{proof}

\section{Existence}
\label{section: exist}
In this section, we are going to prove the existence of a solution $u$ to the problem~\eqref{problema}. Let us denote
\[
H_a=\Set{u\in\sbv(\R^n) | \begin{gathered}
u(x)=1 \text{ in }\Omega \\
u(x)\in\set{0}\cup[a,1] \text{ $\Ln$-a.e.} \\
\supp u \subseteq B_\frac{1}{a^{q-1}}
\end{gathered}}.
\]
We also denote by $H_0$ the set
\[
H_0=\Set{u\in\sbv(\R^n) | \begin{gathered}
u(x)=1 \text{ in }\Omega \\
u(x)\in[0,1] \text{ $\Ln$-a.e.}
\end{gathered}}.
\]
Notice that if $u\in H_0$ is an inward minimizer, by \autoref{teor: lower-bound} and \autoref{cor: density0}, then $u\in H_{\delta_0}$.
\begin{prop}\label{prop:H_0H_a}
Let $u\in H_0$. Then $u$ is a minimizer for the functional~\eqref{funzionale} on $H_0$ if and only if $u\in H_{\delta_0}$ and 
\[
\mathcal{F}(u)\le \mathcal{F}(v) \qquad \forall v\in H_{\delta_0}.
\]
\begin{proof}
As we observed before, if $u$ is a minimizer over $H_0$ then $u$ is in $H_{\delta_0}$, hence it is a minimizer over $H_{\delta_0}$. Conversely, let us take $u\in H_{\delta_0}$ a minimizer over $H_{\delta_0}$, and let us consider in addition $v\in H_0$. Without loss of generality assume $\mathcal{F}(v)\le2\tilde{\mathcal{F}}$. We will prove that there exists a sequence $w_k$ of inward minimizers such that
\[
\mathcal{F}(w_k)\le\mathcal{F}(v)+\frac{C}{k^{q-1}}.
\]
We first construct a family of functions $v_a\in H_a$ such that 
\[
\mathcal{F}(v_a)\le\mathcal{F}(v)+r(a),
\]
with $\lim_{a\to 0}r(a)=0$. Let $0<a<1$, and let ${v_a=v\chi_{\set{v\ge a}\cap B_{\rho(a)}}}$, where $\rho(a)=a^{1-q}$, we have
\begin{equation}
\label{eq: vtrunc}
\begin{split}
\mathcal{F}(v_a)-\mathcal{F}(v)&\le     \beta\int_{\partial^*(\set{v\ge a}\cap B_{\rho(a)})\setminus J_v}v^q\,d\Hn\\[5 pt]
&\le\beta a^q P(\set{v\ge a})+\beta \int_ {(\partial B_{\rho(a)}\cap \set{v\ge a})\setminus J_v} v^q\,d\Hn  \\[5 pt]
&\le \beta a^{q}\left( P(\set{v\ge a})+\frac{1}{a^q}\int_ {(\partial B_{\rho(a)}\cap \set{v\ge a})\setminus J_v} v\,d\Hn\right).
\end{split}
\end{equation}
In order to estimate the right-hand side, fix $t\in(0,1)$, and observe that by the coarea formula
\begin{equation}
\label{eq: coarea}
\int_0^tP(\set{v\ge a})\, da\le \abs{Dv}(\R^n),
\end{equation}
while, with a change of variables,
\[\int_0^t\frac{1}{a^q}\int_ {(\partial B_{\rho(a)}\cap \set{v\ge a})\setminus J_v} v\,d\Hn\,da\le(q-1)\int_0^{+\infty}\int_{\partial B_r\setminus J_v}v \,d\Hn\,dr=(q-1)\norma{v}_{L^1(\R^n)}.\]
\[
\int_0^t\left( P(\set{v\ge a})+\frac{1}{a^q}\int_ {(\partial B_{\rho(a)}\cap \set{v\ge a})\setminus J_v} v\,d\Hn\right)\,da\le q \norma{v}_{\bv}.
\]
By mean value theorem, for every $k\in\mathbb{N}$ we can find $a_k\le 1/k$ such that
\[
P(\set{v\ge a_k})+\frac{1}{a_k^q}\int_ {(\partial B_{\rho(a_k)}\cap \set{v\ge a_k})\setminus J_v} v\,d\Hn\le \frac{q\norma{v}_{\bv}}{a_k},
\]
and in \eqref{eq: vtrunc} we get
\[
\mathcal{F}(v_{a_k})\le \mathcal{F}(v)+q\beta  a_k^{q-1}\norma{v}_{\bv}\le\mathcal{F}(v)+q\beta\frac{\norma{v}_{\bv}}{k^{q-1}}.
\]
We now construct the aforementioned sequence of inward minimizers: let us consider the functional
\[
\mathcal{G}_k(A)=\mathcal{F}(v_{a_k}\chi_A),
\]
with $A$ containing $\Omega$ and contained in $\set{v_{a_k}>0}$. If we consider $A_j$ a minimizing sequence for $\mathcal{G}_k$, then they are certainly equibounded. Moreover, 
\[
\begin{split}
\mathcal{G}_k(A_j)&\ge \Ln(A_j\setminus \Omega)+\beta\int_{J_{\chi_{A_j} v_{a_k}}}\left(\underline{\chi_{A_j} v_{a_k}}^{q}+\overline{\chi_{A_j} v_{a_k}}^{\,q}\right)\,d\Hn \\[5 pt]
&\ge \Ln(A_j)+\beta a_k^q \Hn(J_{\chi_{A_j} v_{a_k}})-\Ln(\Omega).
\end{split}
\]
Notice in addition that since $v_{a_k}\ge a_k$ on its support, then the jump set $J_{\chi_{A_j} v_{a_k}}$ clearly contains $\partial^* A_j$. We now have that $\chi_{A_j}$ satisfies the conditions of \autoref{teor:compact}, and eventually extracting a subsequence we can suppose that 
\[
A_j\xrightarrow[]{L^1} A^{(k)},
\]
with a suitable $A^{(k)}$, and moreover, letting $w_k=\chi_{A^{(k)}} v_{a_k}$, we have
\[
\mathcal{F}(w_k)\le\inf_{\Omega\subseteq A \subseteq \set{v_{a_k}>0}}\mathcal{G}_k(A)\le \mathcal{F}(v_{a_k})\le \mathcal{F}(v)+q\beta \frac{\norma{v}_{\bv}}{k^{q-1}}.
\]
By construction $w_k$ is an inward minimizer, therefore we have $w_k\in H_{\delta_0}$, and consequently, we can compare it with $u$, obtaining
\[
\mathcal{F}(u)\le \mathcal{F}(w_k)\le \mathcal{F}(v)+q\beta\frac{\norma{v}_{\bv}}{k^{q-1}}.
\]
Letting $k$ go to infinity we get the thesis.
\end{proof}
\end{prop}

\begin{prop}
\label{teor: exist}
There exists a minimizer for problem~\eqref{problema}.
\end{prop}
\begin{proof}
By \autoref{prop:H_0H_a} and \autoref{teor: lower-bound} it is enough to find a minimizer in $H_{\delta_0}$. Let $u_k$ be a minimizing sequence in $H_{\delta_0}$, then, for $k$ large enough, we have
\[\beta\delta_0^q\Hn(J_{u_k})+\int_{\R^n}\abs{\nabla u_k}^p\,d\Ln\leq \mathcal{F}(u_k)\leq2\tilde{\mathcal{F}}.\]
From \autoref{teor:compact} we have that there exists $u\in H_{\delta_0}$ such that, up to a subsequence, $u_k$ converges to $u$ in $L^1_{\loc}$ and
\[\mathcal{F}(u)\leq\liminf_{k} \mathcal{F}(u_k),\]
therefore $u$ is a solution.
\end{proof}
\begin{proof}[Proof of \autoref{teor: main1}]
The result is obtained by joining \autoref{teor: exist} and \autoref{teor: lower-bound}.
\end{proof}
\section{Density estimates}
\label{section: density}
In this section, we prove the density estimates in \autoref{teor: main2} by adapting techniques used in \cite{CK} analogous to classical ones used in \cite{Giorgi} to prove density estimates for the jump set of almost-quasi minimizers of the Mumford-Shah functional. 
\begin{defi}
\label{defi: localminimizer}
Let $u\in\sbv(\R^n)$ be a function such that $u=1$ a.e. in $\Omega$. We say that $u$ is a \emph{local minimizer} for $\mathcal{F}$ on a set of finite perimeter $E\subset\R^n\setminus\Omega$, if \[\mathcal{F}(u)\leq\mathcal{F}(v),\] for every $v\in\sbv(\R^n)$ such that $u-v$ has support in $E$.
\end{defi}
Let $E$ be a set of finite perimeter. We introduce the notation
\[\mathcal{F}(u;E)=\int_E \abs{\nabla u}^p\,d\Ln+\beta\int_{J_u\cap E} \left(\overline{u}^q+\underline{u}^q\right)\,d\Hn+\Ln\left(\set{u>0}\cap E\right).\]
To prove \autoref{teor: main2} we will use the following Poincaré-Wirtinger type inequality whose proof can be found in \cite[Theorem 3.1 and Remark 3.3]{Giorgi}.
Let $\gamma_n$  be the isoperimetric constant relative to the balls of $\R^n$, i.e.
\[\min\Set{\Ln(E\cap B_r)^{\frac{n-1}{n}},\Ln(E\setminus B_r)^{\frac{n-1}{n}}}\leq \gamma_n P(E;B_r),\]
for every Borel set $E$, then
\begin{prop}
\label{lem: PW}
Let $p\ge1$ and let $u\in\sbv(B_r)$ such that 
\begin{equation}\label{eq:ipotesidelca} \left(2\gamma_n \Hn(J_u\cap B_r)\right){^\frac{n}{n-1}}<\dfrac{\Ln(B_r)}{2},\end{equation}
 Then there exist numbers $-\infty<    \tau^-\leq m\leq \tau^+<+\infty$ such that the function \[\tilde{u}=\max\set{\min\set{u,\tau^+},\tau^-},\] satisfies
\[\norma{\tilde{u}-m}_{L^p}\leq C \norma{\nabla u}_{L^p}\]
and
\[\Ln(\set{u\neq\tilde{u}})\leq C\left(\Hn(J_u\cap B_r)\right)^{\frac{n}{n-1}},\]
where the constants depend only on $n$, $p$, and $r$.
\end{prop}

\begin{lemma}
\label{lem: density}
Let $u\in H_s$ be a local minimizer on $B_r(x)$ in the sense of definition \autoref{defi: localminimizer}. For sufficiently small values of $\tau$, there exist values $r_0,\eps_0$ depending only on $n,\tau,\beta,p,q$ and $s$ such that, if $r<r_0$, 
\[\Hn(J_u\cap B_r(x))\leq \eps_0 r^{n-1},\]
and
\[\mathcal{F}(u;B_r(x))\geq r^{n-\frac{1}{2}},\]
then
\[\mathcal{F}(u;B_{\tau r}(x))\leq \tau^{n-\frac{1}{2}}\mathcal{F}(u;B_r(x)).\]
\end{lemma}
\begin{proof}
Without loss of generality, assume $x=0$. Assume by contradiction that the conclusion fails, then for every $\tau>0$ there exists a sequence $u_k\in H_s$ of local minimizers on $B_{r_k}$, with $\lim_{k}r_k=0$, such that
\[\dfrac{\Hn(J_{u_k}\cap B_{r_k})}{r_k^{n-1}}=\eps_k,\]
with $\lim_k \eps_k=0$, 
\begin{equation}\label{eq :ipotesi1 lemma}\mathcal{F}(u_k;B_{r_k})\geq r_k^{n-\frac{1}{2}},\end{equation}
and yet
\begin{equation}\label{eq :ipotesi2} \mathcal{F}(u_k;B_{\tau r_{k}})>\tau^{n-\frac{1}{2}}\mathcal{F}(u_k;B_{r_k}).\end{equation}
For every $t\in[0,1]$, we define the sequence of monotone functions
\[\alpha_k(t)=\dfrac{\mathcal{F}(u_k;B_{t r_{k}})}{\mathcal{F}(u_k,B_{r_k})}\le 1.\]
By compactness of $\bv([0,1])$ in $L^1([0,1])$, we can assume that, up to a subsequence, $\alpha_k$ converges $\mathcal{L}^1$-almost everywhere to a monotone function $\alpha$. Moreover, notice that, by~\eqref{eq :ipotesi2}, for every $k$ \begin{equation}\label{eq:alpha>}\alpha_k(\tau)>\tau^{n-\frac{1}{2}}.\end{equation} 
Our final aim is to prove that there exists a $p$-harmonic function $v\in W^{1,p}(B_1)$ such that for every $t$
\[
\lim_{k\to+\infty}\alpha_k(t)=\alpha(t)=\int_{B_t}\abs{\nabla v}^p\,d\Ln.
\]
Let 
\[E_k=r_k^{p-n}\mathcal{F}(u_k;B_{r_k}), \qquad \qquad v_k(x)=\dfrac{u_k(r_k x)}{E_k^{1/p}}.\]
Then $v_k\in\sbv(B_1)$, and
\[\int_{B_1}\abs{\nabla v_k}^p\,d\Ln\leq1,\qquad \qquad\Hn(J_{v_k}\cap B_1)=\eps_k.\]
Thus, applying the Poincaré-Wirtinger type inequality in \autoref{lem: PW} to functions $v_k$ we obtain truncated functions $\tilde{v}_k$ and values $m_k$, such that
\[\int_{B_1}\abs{\tilde{v}_k-m_k}^p\,d\Ln\leq C\]
and
\begin{equation}
\label{eq: noteqmeas}
\Ln(\set{v_k\neq \tilde{v_k}})\leq C\left(\Hn(J_{v_k}\cap B_1)\right)^{\frac{n}{n-1}}\leq C\eps_k^{\frac{n}{n-1}}.
\end{equation}\\
We\marcomm{\textbf{Step 1:}} prove that there exists $v\in W^{1,p}(B_1)$ such that
\[
\tilde{v}_k-m_k\xrightarrow{L^p(B_1)}{v},
\]
\begin{equation}
\label{}
    \int_{B_\rho}\abs{\nabla v}^p\,d\Ln\le\alpha(\rho),\qquad \text{for $\mathcal{L}^1$-a.e. $\rho<1$,}
\end{equation}
and
\begin{equation}
\label{eq:resto}
\lim_k \dfrac{r_k^{p-1}}{E_k}\Hn(\set{v_k\neq\tilde{v}_k}\cap\partial B_\rho)=0, \qquad \text{for $\mathcal{L}^1$-a.e. $\rho<1$.}
\end{equation}
Notice that 
\[
\int_{B_1}\abs{\nabla (\tilde{v}_k-m_k)}^p\,d\Ln\le\int_{B_1}\abs{\nabla v_k}^p\,d\Ln\le 1,
\]
since $\tilde{v}_k$ is a truncation of $v$. From compactness theorems in $\sbv$ (see for instance \cite[Theorem 3.5]{Giorgi}), we have that $\tilde{v}_k-m_k$ converges in $L^p(B_1)$ and $\Ln$-almost everywhere to a function $v\in W^{1,p}(B_1)$, since $\Hn(J_{\tilde{v}_k})$ goes to $0$ as $k\to+\infty$. Moreover, for every $\rho<1$,
\[\int_{B_\rho}\abs{\nabla v}^p\,d\Ln\leq\liminf_{k}\int_{B_\rho}\abs{\nabla \tilde{v}_k}^p\,d\Ln,\]
and
\[\int_{B_\rho}\abs{\nabla v}^p\,d\Ln\leq\liminf_{k}\int_{B_\rho}\abs{\nabla \tilde{v}_k}^p\,d\Ln\leq\liminf_{k}\alpha_k(\rho)=\alpha(\rho),\]
since by definition
\[
\int_{B_\rho}\abs{\nabla v_k}^p\,d\Ln=\frac{r_k^{p-n}}{E_k}\int_{B_{\rho r_k}}\abs{\nabla u_k}^p\,d\Ln\le\frac{r_k^{p-n}}{E_k}\mathcal{F}(u_k;B_{\rho r_k})\le\alpha_k(\rho).
\]
Finally, up to a subsequence,
\[\lim_{k} \dfrac{r_k^{p-1}}{E_k}\Ln(\set{v_k\neq\tilde{v}_k})=0.\]
Indeed, by \eqref{eq: noteqmeas},
\[\dfrac{r_k^{p-1}}{E_k}\Ln(\set{v_k\neq\tilde{v}_k})\leq C\dfrac{r_k^{p-1}}{E_k}\eps_k^{\frac{n}{n-1}},\]
which tends to zero as long as $r_k^{p-1}/E_k$ is bounded. On the other hand, if $r_k^{p-1}/E_k$ diverges, we could use the fact that $\eps_k\leq s^{-q} \mathcal{F}(u_k;B_{r_k})r_k^{1-n}$ and get
\[\dfrac{r_k^{p-1}}{E_k}\Ln(\set{v_k\neq\tilde{v}_k})\leq C \dfrac{r_k^{p-1}}{E_k}\left(\dfrac{E_k}{r_k^{p-1}}\right)^{\frac{n}{n-1}}\]
which goes to zero. Using Fubini's theorem we have \eqref{eq:resto}.
\bigskip

\noindent Let $\tilde{u}_k(x)=E_k^{1/p}\tilde{v}_k(\frac{x}{r_k})$, and for every $t\in[0,1]$ we define
\[\tilde{\alpha}_k(t)=\dfrac{\mathcal{F}(\tilde{u}_k;B_{t r_{k}})}{\mathcal{F}(u_k,B_{r_k})}.\]
The $\tilde{\alpha}_k$ functions are also monotone and bounded: the jump set of $\tilde{u}_k$ is contained in $J_{u_k}$, 
therefore we can write
\[\tilde{\alpha}_k(t)\leq\alpha_k(t)+\dfrac{2\beta\Hn(J_{u_k}\cap B_{t r_{k}})}{\mathcal{F}(u_k;B_{r_k})}\leq\left(1+\dfrac{2}{s^q}\right)\alpha_k(t),\]
using the fact that $u_k\in H_s$. As done for $\alpha_k$, we can assume that $\tilde{\alpha}_k$ converges $\mathcal{L}^1$-almost everywhere to a function $\tilde{\alpha}$.\bigskip

\noindent Let\marcomm{\textbf{Step 2:}} $I\subset [0,1]$ be the set of values $\rho$ for which~\eqref{eq:resto} holds, $\alpha_k$ and $\tilde{\alpha}_k$ converge and $\alpha$ and $\tilde{\alpha}$ are continuous. Notice that $\mathcal{L}^1(I)=1$. Fix $\rho,\rho'\in I$ with $\rho<\rho'<1$ and let
\[\mathcal{I}_k(\xi)=\beta E_k^{q/p-1}r_k^{p-1}\int_{J_{\xi}\cap( B_{\rho'}\setminus B_\rho)}\left(\overline{\xi}^q+\underline{\xi}^q\right)\,d\Hn,\]
with $\xi\in\sbv(B_1)$. Let $w\in W^{1,p}(B_1)$ and consider $\eta$ a smooth cutoff function supported on $B_{\rho'}$ and identically equal to $1$ in $B_\rho$. Let
\[\varphi_k=((w+m_k)\eta+\tilde{v}_k(1-\eta))\chi_{B_{\rho'}}+v_k\chi_{B_1\setminus B_{\rho'}}.\]
We want to prove that
\begin{equation}\label{eq: alpha tilde}\tilde{\alpha}_k(\rho')-\tilde{\alpha}_k(\rho)\ge \int_{B_{\rho'}\setminus B_\rho} \abs{\nabla\tilde{v}_k}^p\,d\Ln+\mathcal{I}_k(\tilde{v}_k),\end{equation}
and
\begin{equation}\label{eq: alpha k}\alpha_k(\rho')\le R_k+\int_{B_{\rho'}}\abs{\nabla \varphi_k}^p\,d\Ln+\mathcal{I}_k(\varphi_k),\end{equation}
where $R_k$ goes to zero as $k$ goes to infinity.
We immediately compute
\[\begin{split}\tilde{\alpha}_k(\rho')-\tilde{\alpha}_k(\rho)=&\mathcal{F}(u_k;B_{r_k})^{-1}\left[\int_{B_{\rho' r_{k}}\cap B_{\rho r_{k}}}\abs{\nabla\tilde{u}_k}^p\,d\Ln+\beta\int_{J_{\tilde{u}_k}\cap (B_{\rho' r_{k}}\setminus B_{\rho r_{k}})}\left(\overline{\tilde{u}_k}^q+\underline{\tilde{u}_k}^q\right)\,d\Hn\right]\\[5 pt]&+\mathcal{F}(u_k;B_{r_k})^{-1}\Ln(\set{\tilde{u}_k>0}\cap (B_{\rho' r_k}\setminus B_{\rho r_k}))\\[5 pt]
\ge& \int_{B_{\rho'}\setminus B_\rho}\abs{\nabla\tilde{v}_k}^p\,d\Ln+E_k^{q/p-1}r_k^{p-1}\beta\int_{J_{\tilde{v}_k}\cap( B_{\rho'}\setminus B_\rho)}\left(\overline{\tilde{v}_k}^q+\underline{\tilde{v}_k}^q\right)\,d\Hn,\end{split}\]
and then we have \eqref{eq: alpha tilde}. Now let
 $\psi_k=E_k^{1/p}\varphi_k(x/r_k)$ and observe that $\psi_k$ coincides with $u_k$ outside $B_{\rho' r_k}$. %
We get from the local minimality of $u_k$ that
\begin{equation}
\label{eq: finproof1}
\begin{split}
\mathcal{F}(u_k;B_{r_k})\le \mathcal{F}(\psi_k;B_{r_k})&=\mathcal{F}(\psi_k;B_{\rho' r_k})+\beta\int_{\set{u_k\ne \tilde{u}_k}\cap\partial B_{\rho' r_k}}\left(\underline{\psi_k}^{q}+\overline{\psi_k}^{\,q}\right)\,d\Hn\\
&\hphantom{=}+\mathcal{F}(u_k;B_{r_k}\setminus \overline{B_{\rho' r_k}}\,)\\[6 pt]
&\le \mathcal{F}(\psi_k;B_{\rho' r_k}) +2\beta  r_k^{n-1}\Hn(\set{v_k\neq\tilde{v}_k}\cap\partial B_{\rho'})\\[3 pt]
&\hphantom{=}+\mathcal{F}(u_k;B_{r_k}\setminus \overline{B_{\rho' r_k}}\,).
\end{split}
\end{equation}
So, in particular,  we have
\[\begin{split}\mathcal{F}(u_k;B_{\rho'r_k})=&\mathcal{F}(u_k;B_{r_k})-\mathcal{F}(u_k;B_{r_k}\setminus\overline{B_{\rho'rfcfc_k}})-\beta\int_{J_{u_k}\cap\partial B_{\rho'r_k}}\left(\overline{u_k}^q+\underline{u_k}^q\right)\,d\Hn\\[7 pt]\le& \,2\beta  r_k^{n-1}\Hn(\set{v_k\neq\tilde{v}_k}\cap\partial B_{\rho'})+\mathcal{F}(\psi_k;B_{\rho' r_k}).\end{split}\]
Dividing by $\mathcal{F}(u_k;B_{r_k})$ and using \eqref{eq:resto} we get
\[
\alpha_k(\rho')\le R_k +r_k^{p-n}E_k^{-1}\mathcal{F}(\psi_k;B_{\rho'r_k}).
\]
With appropriate rescalings we have
\[\begin{split}
r_k^{p-n}E_k^{-1}\mathcal{F}(\psi_k;B_{\rho' r_k})=& \int_{B_{\rho'}}\abs{\nabla\varphi_k}^p\,d\Ln+ \mathcal{I}_k(\varphi_k)+r_k^p E_k^{-1}\Ln(\set{\varphi_k>0}\cap B_{\rho'}).\end{split}\]
From~\eqref{eq :ipotesi1 lemma} and the definition of $E_k$, we have
\[r_k^p E_k^{-1}\Ln(\set{\varphi_k>0}\cap B_{\rho'})\leq \omega_n r_k^{1/2},\]
and then we get \eqref{eq: alpha k}.

\bigskip
\noindent We \marcomm{\textbf{Step 3:}}want to prove that for every $\varphi\in W^{1,p}(B_1)$ such that $v-\varphi$ is supported on $B_{\rho}$, we have
\begin{equation}
\label{eq: alphafin}
\alpha(\rho')\le\int_{B_\rho}\abs{\nabla\varphi}^p\,d\Ln + C\left[\tilde\alpha(\rho')-\tilde\alpha(\rho)\right] + C\int_{B_{\rho'}\setminus B_\rho}\abs{\nabla\varphi}^p\,d\Ln,
\end{equation}
where $C$ does not depend on either $\rho$ or $\rho'$. From the definition of $\varphi_k$, we have that on $B_\rho$
\[\nabla\varphi_k=\nabla w\]
and on $B_{\rho'}\setminus B_\rho$
\[\nabla\varphi_k=\eta\nabla w+(w+m_k-\tilde{v}_k)\nabla\eta+\nabla\tilde{v}_k(1-\eta),\]
so that
\begin{equation}
\label{eq: phikgrad}
\begin{split}\int_{B_{\rho'}}\abs{\nabla\varphi_k}^p\,d\Ln\leq&\int_{B_\rho}\abs{\nabla w}^p\,d\Ln\\[5 pt]
&+C\left[\int_{B_{\rho'}\setminus B_\rho}\abs{\nabla\tilde{v}_k}^p\,d\Ln+\int_{B_{\rho'}\setminus B_\rho}(\abs{\nabla w}^p+\abs{w+m_k-\tilde{v}_k}^p\abs{\nabla\eta}^p)\,d\Ln\right].\end{split}
\end{equation}
We split the proof into two cases: either 
\begin{equation}\label{eq:caso1}\limsup_{k}E_k>0\end{equation}
or
\begin{equation}\label{eq:caso2}\lim_{k}E_k=0.\end{equation}
Assume~\eqref{eq:caso1} occurs. Notice that $s\le u_k\le 1$ for every $k$, then by definition we have that, for every $k$, $s\le E_k^{1/p} \tilde{v}_k\le 1$ and, since $m_k$ is a median of $v_k$, $0\le E_k^{1/p} m_k\le 1$. In particular we have that
\[\abs{\tilde{v}_k-m_k}\le \frac{2}{E_k^{1/p}},\]
passing to the limit when $k$ goes to infinity we have that
\[\norma{v}_{\infty}\le\liminf_{k}\frac{2}{E_k^{1/p}}<+\infty\quad\Ln\text{-a.e.}\]
then $v$ belongs to $L^\infty(B_1)$ and there exists a positive constant $C$ independent of $k$, and a natural number $\overline{k}$ such that
\[\abs{v+m_k-\tilde{v}_k}\le \frac{C}{E_k^{1/p}}\le \frac{C}{s} \tilde{v}_k\quad\Ln\text{-a.e.}\]
for all $k>\overline{k}$. Let $\varphi\in W^{1,p}(B_1)$ with $v-\varphi$  supported on $B_{\rho}$, and let $w=\varphi$ in the definition of $\varphi_k$, then, for every $k>\overline{k}$,  we have 
\begin{equation}\label{eq:stimavarphik}\abs{\varphi_k}=\abs{\tilde{v}_k+(v+m_k- \tilde{v}_k)\eta}\le C \tilde{v}_k\end{equation}
$\Ln$-a.e. on $B_{\rho'}\setminus B_\rho$. From~\eqref{eq:stimavarphik} we have that
\begin{equation}
\label{eq: I_k}
\mathcal{I}_k(\varphi_k)\le C\mathcal{I}_k(\tilde{v}_k).
\end{equation}
Notice, in addition, that \eqref{eq: phikgrad} reads as
\begin{equation}
    \label{eq: phikgrad caso1}
    \begin{split}\int_{B_{\rho'}}\abs{\nabla\varphi_k}^p\,d\Ln\leq&\int_{B_\rho}\abs{\nabla \varphi}^p\,d\Ln\\[5 pt]
&+C\int_{B_{\rho'}\setminus B_\rho}\abs{\nabla\tilde{v}_k}^p\,d\Ln+C\int_{B_{\rho'}\setminus B_\rho}\abs{\nabla \varphi}^p\,d\Ln+ R_k.\end{split}
\end{equation}
finally joining ~\eqref{eq: alpha k},~\eqref{eq: phikgrad caso1},~\eqref{eq: I_k}, and~\eqref{eq: alpha tilde}, we have
\[\alpha_k(\rho')\le \int_{B_\rho}\abs{\nabla\varphi}^p\,d\Ln+C\left[\tilde{\alpha}_k(\rho')-\tilde{\alpha}_k(\rho)\right]+C\int_{B_{\rho'}\setminus B_\rho} \abs{\nabla \varphi}^p\,d\Ln+R_k.\]
Letting $k$ go to infinity we get~\eqref{eq: alphafin}.
\medskip 

Suppose now that \eqref{eq:caso2} occurs. The functions $\abs{\tilde{v}_k-m_k}^p$, $\abs{v}^p$ are uniformly integrable, namely for every $\varepsilon>0$ there exists a $\sigma=\sigma_\varepsilon<\varepsilon$ such that if $A$ is a measurable set with $\abs{A}<\sigma$, then
\begin{equation}
\label{eq: unifint}
\int_{A}\abs{\tilde{v}_k-m_k}^p\,d\Ln+\int_{A}\abs{v}^p\,d\Ln<\varepsilon.
\end{equation}
Since $v\in L^p(B_1)$, we can find $M>1/\varepsilon$ such that 
\begin{equation}
\label{eq: measure}
\abs{\set{\abs{v}>M}}<\sigma.
\end{equation}
Setting $w=\varphi_M=\max\set{-M,\min\set{\varphi,M}}$, then \eqref{eq: phikgrad} reads as
\begin{equation}
\label{eq: phikgrad caso2}
\begin{split}\int_{B_{\rho'}}\abs{\nabla\varphi_k}^p\,d\Ln\leq&\int_{B_\rho\cap\set{\abs{\varphi}\le M}}\abs{\nabla \varphi}^p\,d\Ln+ C\int_{\left(B_{\rho'}\setminus B_\rho\right)\cap\set{\abs{\varphi}\le M}}\abs{\nabla \varphi}^p\,d\Ln\\[5 pt]
&+C\left[\int_{B_{\rho'}\setminus B_\rho}\abs{\nabla\tilde{v}_k}^p\,d\Ln+\int_{B_{\rho'}\setminus B_\rho}\abs{\varphi_M+m_k-\tilde{v}_k}^p\abs{\nabla\eta}^p\,d\Ln\right].\end{split}
\end{equation}
We can estimate the last integral as follows
\begin{equation}
\label{eq: infinitesimalgrad}
\begin{split}
\int_{B_{\rho'}\setminus B_\rho}\abs{\varphi_M+m_k-\tilde{v}_k}^p\abs{\nabla\eta}^p\,d\Ln &\le C\eps + \int_{\left(B_{\rho'}\setminus B_\rho\right)\cap\set{\abs{v}\le M}}\abs{v+m_k-\tilde{v}_k}^p\abs{\nabla\eta}^p\,d\Ln].\\[7 pt]
&=C\eps+R_k,
\end{split}
\end{equation}
where we used \eqref{eq: measure} and \eqref{eq: unifint}, and $C$ only depends on $\rho$ and $\rho'$. Furthermore, we have
\begin{equation}
    \label{eq: jumpestim caso2}
    \mathcal{I}_k(\varphi_k)\le R_k + C\mathcal{I}_k(\tilde{v}_k).
\end{equation}
Indeed, as before, $\abs{\tilde{v}_k-m_k}\le C \tilde{v}_k$, while 
\[
\begin{split}
E_k^{q/p-1}r_k^{p-1}\int_{J_{\tilde{v}_k}\cap\left(B_{\rho'}\setminus B_\rho\right)}\abs{\varphi_M}^q\,d\Hn&\le M^q E_k^{q/p-1}r_k^{p-1}\Hn\left({J_{\tilde{v}_k}\cap\left(B_{\rho'}\setminus B_\rho\right)}\right)\\
&\le M^q E_k^{\frac{q}{p}}\frac{r_k^{p-1}\varepsilon_k}{E_k}\\
&\le \frac{M^q}{s^q}E_k^{\frac{q}{p}},
\end{split}
\]
which goes to 0 as $k\to\infty$. Finally, joining \eqref{eq: alpha k}, \eqref{eq: phikgrad caso2}, \eqref{eq: infinitesimalgrad}, \eqref{eq: jumpestim caso2}, and \eqref{eq: alpha tilde}, we have
\[
\alpha_k(\rho')\le R_k + \int_{B_\rho\cap\set{\abs{\varphi}\le M}}\abs{\nabla \varphi}^p+C\left[\tilde\alpha(\rho')-\tilde\alpha(\rho)\right]+C\int_{\left(B_{\rho'}\setminus B_\rho\right)\cap\set{\abs{\varphi}\le M}}\abs{\nabla \varphi}^p\,d\Ln+C\eps.
\]
Taking the limit as $k$ tends to infinity, and then the limit as $\eps$ tends to 0, we get \eqref{eq: alphafin}.

\bigskip
\noindent We are now in a position to prove that $v$ is $p$-harmonic: taking the limit as $\rho'$ tends to $\rho$ in \eqref{eq: alphafin}, we have that if $\varphi\in W^{1,p}(B_1)$, with $v-\varphi$ supported on $B_\rho$, 
\[\int_{B_{\rho}}\abs{\nabla v}^p\,d\Ln\leq\alpha(\rho)\leq\int_{B_{\rho}}\abs{\nabla\varphi}^p\,d\Ln,\]
for every $\rho\in I$, therefore $v$ is $p$-harmonic in $B_1$. Notice that this implies that $v$ is a locally Lipschitz function (see \cite[Theorem 7.12]{bv})
.  Moreover, if $\varphi=v$, we have
\[\int_{B_{\rho}}\abs{\nabla v}^p\,d\Ln=\alpha(\rho)\]
for every $\rho\in I$, so that $\alpha$ is continuous on the whole interval $[0,1]$, $\alpha(1)=1$ and $\alpha(\tau)=\lim_k\alpha_k(\tau)\geq \tau^{n-1/2}$. Nevertheless, if $\tau$ is sufficiently small this contradicts the fact that $v$ is locally Lipschitz, since
\[
\tau^{n-\frac{1}{2}}\le\int_{B_{\tau}}\abs{\nabla v}^p\,d\Ln\le C\, \tau^{n}, 
\]
where $C$ is a positive constant depending only on $n$ and $p$.
\end{proof}
\begin{proof}[Proof of \autoref{teor: main2}] Let $u$ be a minimizer for the problem~\eqref{problema}. By \autoref{cor: density0} there exist two positive constants $C(\Omega,\beta,p,q), C_1(\Omega,\beta,p,q)$ such that if $B_r(x)\subseteq\R^n\setminus\Omega$, then
\[
\Hn(J_u\cap B_r(x))\leq C(\Omega,\beta,p,q)r^{n-1},
\]
and if $x\in K_u$
\[\Ln(B_r(x)\cap\set{u>0})\ge C_1(\Omega,\beta,p,q) r^n.\]
We now prove that there exists a positive constant $c=c(\Omega,\beta,p,q)$ such that 
\begin{equation}\label{eq: bound densità sotto}\Hn(J_u\cap B_r(x))\ge c(\Omega,\beta,p,q) r^{n-1}\end{equation}
for every $x\in K_u$ and $B_r(x)\subset\R^n\setminus\Omega$. Assume by contradiction that there exists $x\in J_u$ such that, for $r>0$ small enough,
\[\Hn\left(J_u\cap B_r(x)\right)\le\eps_0 r^{n-1},\]
where $\eps_0$ is the one in \autoref{lem: density}. Iterating \autoref{lem: density} it can be proven (see \cite[Theorem 5.1]{CK}) that
\[\lim_{r\to0^+} r^{1-n}\mathcal{F}(u;B_r)=0,\]
which, in particular, implies
\begin{equation}\label{not in j_u}\lim_{r\to0^+} r^{1-n}\left[\int_{B_r(x)}\abs{\nabla u}^p\,d\Ln+\Hn\left(J_u\cap B_r(x)\right)\right]=0.\end{equation}
By \cite[Theorem 3.6]{Giorgi}, \eqref{not in j_u} implies that $x\notin J_u$, which is a contradiction. Finally, if $x\in K_u$ and
\[\Hn(J_u\cap B_{2r}(x))\leq \eps_0 r^{n-1},\]
there exists $y\in J_u\cap B_r(x)$ such that
\[\Hn\left(J_u\cap B_r(y)\right)\le\eps_0 r^{n-1}\]
which, again, is a contradiction. Then the assertion is proved. The density estimate~\eqref{eq: bound densità sotto} implies in particular that 
\[K_u\subset\Set{x\in\R^n| \limsup_{r\to0^+}\,r^{1-n}\left[\int_{B_r(x)}\abs{\nabla u}^p\,d\Ln+\Hn\left(J_u\cap B_r(x)\right)\right]>0},\]
hence $\Hn(K_u\setminus J_u)=0$ (see for instance \cite[Lemma 2.6]{Giorgi}).
\end{proof}
\begin{oss}
Let $u$ be a minimizer for problem \eqref{problema}, then from  \autoref{teor: lower-bound} we have that the function $u^*=(\beta \delta^q)^{-1/p}u$ is an almost-quasi minimizer for the Mumford-Shah functional
\[MS(v)=\int_{\R^n}\abs{\nabla v}^p\,d\Ln+\Hn(J_v)\]
with the Dirichlet condition $u^*=(\beta \delta^q)^{-1/p}$ on $\Omega$. If $\Omega$ is sufficiently smooth we can apply the results in \cite{BucGia15} to have that the density estimate for the jump set of minimizers holds up to the boundary of $\Omega$.
\end{oss}
\newpage
\nocite{*}
\printbibliography[heading=bibintoc]
\newpage
\Addresses
\end{document}